\documentclass[12pt,a4paper]{article}
\usepackage{fullpage, amsfonts, amsthm, amsmath, setspace, graphicx, amssymb, float, colonequals, hyperref, float,mathrsfs}
\usepackage{comment}
\usepackage[numbers,sort&compress]{natbib}

\usepackage{url}

\newtheorem{conjecture}{Conjecture}[section]
\newtheorem{theorem}{Theorem}[section]
\newtheorem{lemma}[theorem]{Lemma}

\newcommand{\ds}{\displaystyle}

\numberwithin{subcase}{case}
\numberwithin{subsubcase}{subcase}
\numberwithin{claim}{theorem}

\title{\bf Some Recent Results on the Geometry of Complex Polynomials: The Gauss--Lucas Theorem, Polynomial Lemniscates, Shape Analysis, and Conformal Equivalence.}

\author{Trevor J. Richards\thanks{Email: trichards@monmouthcollege.edu}\vspace{6mm}\\{\em Mathematics, Statistics, and Computer Science, Monmouth College}\\{\em Monmouth, IL United States}}

\begin{document}

\maketitle

\begin{center}
\end{center}

\begin{center}
\end{center}

\begin{abstract}
In this article, we survey the the recent literature surrounding the geometry of complex polynomials. Specific areas surveyed are i) Generalizations of the Gauss--Lucas Theorem, ii) Geometry of Polynomials Level Sets, and iii) Shape Analysis and Conformal Equivalence.
\end{abstract}

\section{Introduction}
The geometry of complex polynomials has been an area of ongoing interest since the complex numbers were first conceived of geometrically. Foremost in the historical study of the geometry of complex polynomials has been the problem of finding the zeros (and critical points) of a given polynomial, or failing that, regions guaranteed to contain all (or some or none) of the zeros (or critical points) of the polynomial. The foundational result in this area is the Gauss--Lucas theorem, which states that the critical points of a complex polynomial lie in the convex hull of the zeros of that polynomial. In Section~\ref{sect: GL-related.}, we will survey results which are related to the Gauss--Lucas Theorem.

A natural generalization of the notion of a zero of a complex polynomial $p(z)$ is a lemniscate of $p(z)$. The lemniscates of $p(z)$ are the components of the level sets $$\Lambda_{\epsilon}(p)=\{z:|p(z)|=\epsilon\}$$ for any $\epsilon\in(0,\infty)$ (of course, if $\epsilon=0$ we have reproduced the zero set of $p$). The study of the geometry specifically of these lemniscates also has a long history, dating to the investigation of Cassini and Bernoulli (see~\cite{Yates} for example). In Section~\ref{sect: Geometry of Lemniscates.}, we will survey recent results regarding the geometry of lemniscates, both individually and viewed as a complex of nested curves (nested in the sense of one lying in a bounded component of the complement of another).

Hilbert's theorem, to the effect that the lemniscates of complex polynomials may be used to approximate simple closed curves arbitrarily well, has made these lemniscates a valuable tool in the emerging field of shape analysis. A shape $\Gamma$ is a simple closed path which i) is smooth and ii) divides $\hat{\mathbb{C}}$ into two simply connected domains, one bounded (called $\Omega_+$) and one unbounded (called $\Omega_-$). The fingerprint $\tau:\mathbb{T}\to\mathbb{T}$ of $\Gamma$ is the orientation preserving biholomorphism of the unit circle onto itself obtained by composing the appropriate (ie. subject to certain normalizations) Riemann maps for $\Omega_+$ and $\Omega_-$ in the appropriate way. In Section~\ref{sect: Fingerprints of Shapes and Conformal Equivalence.}, we will introduce these notions of shape and fingerprint properly, and survey recent results relating to the fingerprints of polynomial lemniscates.

In the special case that the shape $\Gamma$ is a proper lemniscate of a complex polynomial $p$ (that is, a lemniscate containing all of the zeros of the polynomial in its bounded face), the fingerprint of $\Gamma$ has a particularly nice form, since the Riemann map for the region $\Omega_+$ may be taken to be $p(z)^{1/n}$ (where $n$ is the degree of $p$). In~\cite{EbenfeltKhavinsonShapiro}, it was shown that the fingerprint of any such $\Gamma$ is the $n^{\text{th}}$ root of a degree-$n$ Blaschke product $B(z)$, and conversely that the $n^{\text{th}}$ root of any degree $n$ Blaschke product is the fingerprint for some proper polynomial lemniscate. From this it follows that for any finite Blaschke product $B$, there is some injective analytic map $\varphi:\mathbb{D}\to\mathbb{C}$, and some complex polynomial $p$ (with $\deg(p)=\deg(B)$) such that $B=p\circ\varphi$ on $\mathbb{D}$. In other words, $B$ is conformally equivalent to a complex polynomial (of the same degree as $B$) on $\mathbb{D}$. This fact has been re-proven by various methods by several authors, and generalized beyond the realm of finite Blaschke products. Also in Section~\ref{sect: Fingerprints of Shapes and Conformal Equivalence.}, we will survey recent results regarding conformal equivalence of arbitrary analytic functions to polynomials (and meromorphic functions to rational functions).

The results contained herein are largely restricted to those appearing in the last ten or so years. Many theorems mentioned below appear in articles with other interesting results not mentioned here. The subjects described above are chosen largely for their appeal to the author's interest, and many results have appeared in other areas related to the geometry of complex polynomials.

%https://www.semanticscholar.org/paper/SOME-RESULTS-ON-THE-GEOMETRY-OF-THE-ZEROS-OF-D%C3%ADaz-Barrero-Egozcue/2af17eac0e3cea5227e364b5acc6943bc1290b57

%https://link.springer.com/chapter/10.1007/978-3-0348-0142-3_16

\section{Gauss--Lucas Related Theorems}
\label{sect: GL-related.}

There continue to be very many contributions the the classical study of the geometry of complex polynomials, which is chiefly concerned with the relations between the zeros, critical points, and coefficients of a complex polynomial. In this section, we will focus on generalizations of the Gauss--Lucas theorem.

\subsection{The Shrinking Hulls of the Zeros of the Derivatives}

If we let $H(p)$ denote the convex hull of the roots of a degree $n$ polynomial $p$, then the sequence $H(p)\supset H(p')\supset\cdots\supset H(p^{(n-1)})$ shrinks to a single point. In 2018, M.~Ravichandran~\cite{Ravichandran} quantified the rate at which this nested sequence shrinks with the following theorem.

\begin{theorem}
For any complex polynomial $p$ and any $r\in(1/2,1)$, $$m\left(H(p^{(\lceil r\deg(p)\rceil)})\right)\leq4(r-r^2)m(H(p)).$$
\end{theorem}

\subsection{Convex Combinations of Incomplete Polynomials}

For $n$ not necessarily distinct points $z_1,z_2,\ldots,z_n\in\mathbb{C}$, and any $1\leq k\leq n$, let $g_k$ denote the $k^{\text{th}}$ incomplete polynomial $$g_k(z)=\ds\prod_{\stackrel{1\leq j\leq n}{j\neq k}}(z-z_j)$$ (that is, the monic degree $n-1$ polynomial whose zeros are exactly $z_1,z_2,\ldots,z_n$, except for $z_k$). In 2008, J.~L.~Diaz-Barrero and J.~J.~Egozcue~\cite{DiazBarreroEgozcue} provided the following generalization of the Gauss--Lucas theorem for convex combinations of incomplete polynomials.

\begin{theorem}
Let $z_1,z_2,\ldots,z_n\in\mathbb{C}$ be not necessarily distinct complex numbers. Let $r_1,r_2,\ldots,r_n\in[0,1]$ satisfy $\sum r_k=1$. Then the roots of the degree $n-1$ polynomial $A(z)=\ds\sum r_kg_k(z)$ all lie in the convex hull of the points $z_1,z_2,\ldots,z_n$.
\end{theorem}

Note that in the previous theorem, setting each $r_k=1/n$ reproduces the classical Gauss--Lucas theorem.

\subsection{Approximate and Assymptotic Gauss--Lucas Theorems}

For a set $K\subset\mathbb{C}$, and an $\epsilon>0$, define $K_\epsilon$ to be the $\epsilon$-neighborhood of $K$. For a polynomial $p$, let $Z(p,K)$ denote the number of zeros of $p$ which lie in $K$. In 2016, V.~Totik~\cite{Totik} established the following asymptotic version of the Gauss--Lucas theorem.

\begin{theorem}\label{thm: Asymptotic GL theorem.}
For any bounded convex set $K\subset\mathbb{C}$, any $\epsilon>0$, and any sequence of polynomials $\{p_n\}$ with $\deg(p_n)=n$, if $\dfrac{Z(p_n,K)}{n}\to1$ then $\dfrac{Z({p_n}',K_\epsilon)}{n-1}\to1$.

\end{theorem}

While not quite fitting in this section, results which are somewhat similar to Theorem~\ref{thm: Asymptotic GL theorem.} in flavor, to the effect that the critical points of a random polynomial converge in distribution to the zeros of the polynomial, may be found in~\cite{ORourkeWilliams,PemantleRivin}.

In 2017, T.~J.~Richards conjectured in a document posted on \textit{arxiv.org}~\cite{Richards1} that underlying the asymptotic Theorem~\ref{thm: Asymptotic GL theorem.} is the following static principle.

\begin{conjecture}\label{conj: Approx. GL theorem.}
For any bounded convex set $K\subset\mathbb{C}$, and any $\epsilon>0$, there is a constant $C_{K,\epsilon}\in(0,1)$ such that, for any polynomial $p$ with sufficiently large degree, if $\dfrac{Z(p,K)}{\deg(p)}>C_{K,\epsilon}$, then $Z(p',K_\epsilon)\geq Z(p,K)-1$.
\end{conjecture}

In 2019, T. J. Richards and S. Steinerberger~\cite{RichardsSteinerberger} proved a weaker version of Conjecture~\ref{conj: Approx. GL theorem.}

\begin{theorem}
For any convex bounded set $K\subset\mathbb{C}$, and any $\epsilon>0$, there is a constant $D_{K,\epsilon}>0$ such that, for any polynomial $p$, if $\dfrac{Z(p,K)}{\deg(p)}>\dfrac{\log(\deg(p))-D_{K,\epsilon}}{\log(\deg(p))}$, then $Z(p',K_\epsilon)\geq Z(p,K)-1$.
\end{theorem}

We also note here that in private correspondence, V.~Totik has communicated a proof of Conjecture~\ref{conj: Approx. GL theorem.} to the author, along with bounds on the constant $C_{k,\epsilon}$, and we look forward to seeing these results in print soon.

\subsection{A non-convex Gauss--Lucas Theorem for Polynomials with Non-negative Coefficients}

In 2017, Bl.~Sendov and H.S.~Sendov\cite{SendovSendov} proved an analogue to the Gauss--Lucas theorem for non-convex sectors of $\mathbb{C}$, provided that the coefficients of the polynomial in question are real and non-negative. In order to state the theorem, for $\alpha\in[0,\pi]$, define $\operatorname{Sect}(\alpha)=\{z\in\mathbb{C}:|\arg(z)|\geq\alpha\}$.

\begin{theorem}
If $p(z)$ has all real and non-negative coefficents, and all of the zeros of $p$ lie in the sector $\operatorname{Sect}(\alpha)$ for any $\alpha\in[0,\pi]$, then all of the critical points of $p$ lie in $\operatorname{Sect}(\alpha)$.
\end{theorem}

\subsection{Converses to the Gauss--Lucas Theorem}

In 2014, N.~Nikolov and B.~Sendov~\cite{NikolovSendov} proved the following converse to the Gauss--Lucas theorem, showing that differentiation is the only non-trivial linear operator which contracts zero sets.

\begin{theorem}
Let $S:\mathbb{C}[z]\to\mathbb{C}[z]$ be a linear operator for which $H(S(p))\subset H(p)$ for all $p\in\mathbb{C}[z]$. Then either $S$ is complex-valued (ie. $S(\mathbb{C})\subset\mathbb{C}$), or there is some $c\in\mathbb{C}\setminus\{0\}$ and some integer $n\geq0$ for which $S(p)=cp^{(n)}$.

\end{theorem}

Another direction in which one might look for a converse to the Gauss--Lucas theorem is an identification of those collection of $n-1$ points in a given convex set $K\subset\mathbb{C}$ which might appear as the critical points of some degree-$n$ complex polynomial having all of its zeros lying in $K$. Along these lines, in 2017 C.~Frayer~\cite{Frayer1} established the following theorem for polynomials with three distinct roots. Any such polynomial may be normalized to have a zero at $1$, and its other two roots on the unit circle. Let $p(z)=(z-1)^k(z-d_1)^m(z-d_2)^n$, for $d_1,d_2\in\mathbb{T}$. Let $P(k,m,n)$ denote the collection of all such polynomials $p$. For $r\in(0,1)$, define $$T_r=\left\{z\in\mathbb{C}:\left|z-\left(1-\dfrac{r}{2}\right)\right|=\dfrac{r}{2}\right\},$$ the circle centered at $1-\dfrac{r}{2}$, which is tangent to the unit circle at $1$.

\begin{theorem}\label{thm: Three distinct zeros.}
Fix positive integers $k$, $m$, and $n$.
\begin{itemize}
\item
No polynomial $p\in P(k,m,n)$ has a critical point on the region interior to $T_{\frac{2k}{k+m+n}}$.

\item
If $m\neq n$, then additionally, no $p\in P(k,m,n)$ has a critical point on the region $D$ defined directly after this theorem.

\item
If $c\in\mathbb{D}$ is in neither of the regions mentioned above, then there is a $p\in P(k,m,n)$ with a critical point at $c$. If $c$ is on the boundary of these regions, this polynomial is unique. If $c$ is not on the boundary of these regions, there are exactly two such polynomials.
\end{itemize}

\end{theorem}

If $m\neq n$ (with $m<n$), as in the second part of the above theorem, the region $D$ is bounded by the degree-$2$ algebraic curve which is parameterized by $\gamma(t)=(x,y)$, where $$x=\dfrac{(m+n+k)^2t^2-[2(m+n+k)(m+n+2k)-4mn]t+4k(m+n+k)}{(m+n+k)((m+n+k)t-2k)(t-2)},$$ and $$y^2=(1-x)(t-1+x),$$ for $\dfrac{2(m+k)}{m+n+k}\leq t\leq\dfrac{2(n+k)}{m+n+k}$.

In~\cite{Frayer2}, Frayer provided a geometric construction of the polynomials whose existence is guaranteed by Theorem~\ref{thm: Three distinct zeros.}.

A monograph could easily be devoted to the many refinements, generalizations, and other work which the Gauss--Lucas theorem has inspired over the years. Those contained in this section are only the most recent ones which refer to general polynomials of one complex variable. For other refinements and generalizations, see~\cite{ORourkeWilliams} and the many references contained therein.

\section{The Geometry of the Lemniscates}
\label{sect: Geometry of Lemniscates.}

The lemniscates of complex polynomials (that is, the components of the level sets $\Lambda_\epsilon(p)=\{z:|p(z)|=\epsilon\}$ for a complex polynomial $p$ and an $\epsilon>0$) have inspired considerable interest since 1680, when they were studied by G.~D.~Cassini (see~\cite{Yates} for example). The length, area circumscribed, convexity, and other geometric properties, provide a common locus of study. In their 1958 paper \textit{Metric Properties of Polynomials}~\cite{ErdosHerzogPiranian}, P.~Erd\H{o}s~et.~al. posed a number of problems surrounding these geometric properties, one of which we will begin with.

\subsection{The Erd\H{o}s--Herzog--Piranian Lemniscate Problem}

Let $L_n$ denote the maximum length of the level set $\Lambda_1(p)$, for any degree $n$ polynomial $p$. Erd\H{o}s~et.~al. conjectured that $p_n(z)=z^n+1$ is the polynomial which maximizes this length: $\Lambda_1(p_n)=L_n$. Note that the maximal length $L_n$ is known to be achieved by some polynomial, and that $\Lambda_1(p_n)$ is known to equal $2n+O(1)$ (for these and other results on the so-called Erd\H{o}s--Herzog--Piranian Lemniscate Problem, see references in~\cite{WangPeng,KuznetsovaTkachev}). In 2003, O.~S.~Kuznetsova and V.~G.~Tkachev~\cite{KuznetsovaTkachev} studied the growth of the length functions for the level sets of analytic functions. They established a number of results, of which we mention one here in the context of complex polynomials. For a rectifiable plane curve $\mathcal{C}$, let $\ell(\mathcal{C})$ denote the length of $\mathcal{C}$. For a monic polynomial $p$, define the auxiliary function $$F_p(t)=\ln\left|\ell(\Lambda_{e^t}(p))\right|-\dfrac{t}{\deg(p)}.$$ Kuznetzova and Tkachev showed the following.

\begin{theorem}\label{thm: Lemniscate growth.}
For a monic polynomial $p\in\mathbb{C}[z]$, the maps $t\mapsto\ell(\Lambda_{e^t}(p))$ and $t\mapsto F_p(t)$ are continuous for $t\in(-\infty,\infty)$. Moreover if $p$ is non-trivial (ie. not of the form $c(z-a)^n$ for constants $c,a\in\mathbb{C}$ and non-negative integer $n$), then the following hold.

\begin{itemize}
\item
Both maps mentioned above are strictly convex on any interval $(s_1,s_2)\subset\mathbb{C}$ not containing the logarithm of the modulus of a critical value of $p$.

\item
$\ds\lim_{t\to+\infty} F_p(t)=\ln(2\pi)$.
\end{itemize}

\end{theorem}

If, for $s\in\mathbb{R}$, we define the dilation function $p_s(z)=e^{-s}p\left(ze^{s/\deg(p)}\right)$, the level set and growth function discussed in Theorem~\ref{thm: Lemniscate growth.} satisfy the following invariance properties: $$\Lambda_{e^t}(p_s)=e^{-s/n}\Lambda_{e^{s+t}}(p)\text{ and }F_{p_s}(t)=F_p(s+t).$$ These properties allow one, for instance, to derive from Theorem~\ref{thm: Lemniscate growth.} also some length properties of the family of level sets $\Lambda_1(z^n+r)$, for $r\in\mathbb{R}$. These length properties were also discovered (and augmented) independently in 2006 by C.~Wang and L.~Peng~\cite{WangPeng}. They showed the following.

\begin{theorem}
For any integer $n\geq 1$ and $r\in\mathbb{R}$, define $\gamma_n(r)=\ell\left(\Lambda_1(z^n+r)\right)$.

\begin{itemize}
\item
${\gamma_n}'\geq 0$ on $(0,1)$ and ${\gamma_n}'\leq 0$ on $(1,\infty)$.

\item
${\gamma_n}''\geq 0$ on $(0,1)\cup(1,\infty)$.

\item
For any integer $n$, $4\log(2)\leq\gamma_n(1)-2n\leq 2(\pi-1)$.
\end{itemize}
\end{theorem}

In 2012, O.~N.~Kosukhin\cite{Kosukhin} gave the following upper bound for $L_n$, which improved on the then best results.

\begin{theorem}
For all $n\geq 2$, $L_n\leq\pi\left(n+\dfrac{25}{23}\right)+\pi\sqrt{\dfrac{n-1}{2}\ln(\pi n(n-1))}$.
\end{theorem}

In 2008, A.~Fryntov and F.~Nazarov~\cite{FryntovNazarov} showed that $p(z)=z^n+1$ locally maximizes the length of the level set $\Lambda_1(p)$, and provided another asymptotic upper-bound for maximal length $L_n=\ds\max_{\deg(p)=n}\ell(\Lambda_1(p))$.

\begin{theorem}
Let $n$ be a positive integer. There is some $\epsilon>0$ such that for any degree $n$ polynomial $p$, if the coefficients of $q(z)=p(z)-(z^n+1)$ are all smaller than $\epsilon$, then $$\ell(\Lambda_1(p))\leq\ell(\Lambda_1(z^n+1)).$$
\end{theorem}

\begin{theorem}
$L_n\leq 2n+o(n)$.
\end{theorem}

\subsection{Regions Bounded by Lemniscates}

We now turn to recent area results for lemniscates. In this section, we use the notation $\Lambda_\epsilon(p)$ to denote also the region circumscribed by the lemniscate $\Lambda_\epsilon(p)$. In 2007, H.~H.~Cuenya and F.~E.~Levis~\cite{CuenyaLevis} proved the following. In the following, for $r>0$ let $M_r$ denote the collection of polynomials for which the minimum distance between any two distinct zeros of $p$ is at least $r$ times the diameter of the zero set of $p$.

\begin{theorem}\label{conj: Disk in lemniscate.}
For any $r>0$, there is a constant $C>0$ such that for any $s>0$ and any polynomial $p\in M_r$, the region $\Lambda_s(p)$ contains a disk $D$ with area at least $m(D)\geq \dfrac{m\left(\Lambda_s(p)\right)}{C}$.
\end{theorem}

Cuenya and Levis conjectured that the suitability condition $p\in M_r$ can be removed from the statement of Theorem~\ref{conj: Disk in lemniscate.}, and proved this in the special case that $p$ has at most three distinct zeros. In 2009, A.~Y.~Solynin and A.~S.~Williams~\cite{SolyninWilliams} established Cuenya and Levis' conjecture, but with dependence in the constant $C$ on the degree of the polynomial $p$.

The following theorem, published by P.~Ding in 2018~\cite{Ding}, relates the area of the region between two lemniscates, the lengths of the two lemniscates, and the curvature of the interceding lemniscates. Let $p$ be a complex polynomial. For $0<r<s$, let $\lambda_r$ and $\lambda_s$ be components of the lemniscates $\Lambda_r(p)$ and $\Lambda_s(p)$, such that $\lambda_r$ lies in a bounded component of ${\lambda_s}^c$. For any $t\in(r,s)$, let $\lambda_t$ denote the components of $\Lambda_t(p)$ which lie in the region $D$ between $\lambda_r$ and $\lambda_s$. Finally, let $\kappa(z)$ denote the curvature at $z$ of the lemniscate of $p$ containing $z$.

\begin{theorem}
Given the notation in the preceding paragraph, the following holds.

\begin{itemize}
\item
The area of $D$ is $\ds\int_r^s\left(\int_{z\in\lambda_t}\dfrac{1}{|p'(t)|}|dz|\right)dt$.

\item
$\ds\int_r^s\ell(\lambda_t)dt=\iint_D|p'(z)|dA$.

\item
$\ell(\lambda_s)=\ell(\lambda_r)+\ds\iint_D\kappa(z)dA$.
\end{itemize}

\end{theorem}

\subsection{Area of and Roundness of the Preimage Under a Polynomial}

In 2004, E.~Crane~\cite{Crane} proved the following results regarding the preimages of measurable sets in the plane under a complex polynomial.

\begin{theorem}
Let $K\subset\mathbb{C}$ be measurable, and let $p$ be a complex polynomial with degree $n$. Then $$m\left(p^{-1}(K)\right)\leq\pi\left(\dfrac{m(K)}{\pi}\right)^{1/n}.$$
\end{theorem}

If the logarithmic capacity of $K\subset\mathbb{C}$ is denoted $cap(K)$, and the roundness of $K$ is defined to be $\rho(K)=\dfrac{m(K)}{\pi cap(K)}$, Crane also proved the following.

\begin{theorem}
Let $K\subset\mathbb{C}$ be measurable, and let $p$ be a complex polynomial with degree $n$. Then $$\rho\left(p^{-1}(K)\right)=\rho(K)^{1/n}.$$
\end{theorem}

\subsection{Lengths of Lemniscates of Random Polynomials}

In 2017, E.~Lundberg and K.~Ramachandran\cite{LundbergRamachandran} studied the lengths of the lemniscates $\Lambda_1(p_n)$, for a random sequence polynomials $\{p_n\}$, proving the following.

\begin{theorem}
Let $\{p_n\}$ be a sequence of complex polynomials, where the coefficients of $p_n(z)=\ds\sum_{j=0}^nc_jz^j$ are chosen i.i.d. with the standard Gaussian density $\dfrac{1}{\pi}exp\left(-|z|^2\right)$. Then $$\ds\lim_{n\to\infty}\mathbb{E}\ell(\Lambda_1(p_n))=C,$$

where $C$ is a constant defined by an integral, numerically determined to be $C\approx8.3882$.
\end{theorem}

\subsection{The Lemniscate Tree of a Polynomial}

In 1991, F.~Catanese and M.~Paluszny~\cite{CatanesePaluszny} published an article exploring the topological configuration of all the lemniscates of a complex polynomial. It follows directly from the maximum modulus principle that the non-critical (also called non-singular) lemniscates (ie. those not containing a critical point) of a polynomial interpolate smoothly between any two critical (or singular) lemniscates, and conversely that if any two critical lemniscates are incomparable (in the sense that neither lies in a bounded component of the complement of the other), then these two critical lemniscates lie in different bounded components of the complement of some third critical lemniscate. It follows that the topology of the graph $y=|p(z)|$ may be entirely determined by knowing the configuration of only the critical lemniscates. To each such configuration, Catanese and Paluszny associated a tree, whose nodes represent the distinct critical points and zeros of the polynomial, with an edge between two nodes $a$ and $b$ if $a$ represents a non-trivial critical point, and the zero or critical point of $p$ which is represented by $b$ lies in a bounded face of the critical lemniscate containing the critical point represented by $a$, or vice versa. Catanese and Paluszny showed that there is a bijection between the collection of simple central balanced binary trees and the connected components of the space of lemniscate-generic complex polynomials (ie. those complex polynomials all of whose critical values have different moduli).

In forthcoming work, M.~Epstein~et.~al.~\cite{EpsteinHaninLundberg} analyzed the lemniscate tree of random polynomials. They established the following theorem, where $LT_n$ denotes the collection of generic lemniscate trees with $n$ leaves (ie. those trees corresponding to lemniscate-generic degree-$n$ complex polynomials). They additionally identified the out-degree of a node in a lemniscate tree as $0$ if the node represents a zero of the polynomial, and $2$ if the node represents a non-trivial critical point of the polynomial.

\begin{theorem}
Let $\{T_n\}\subset LT_n$ be a sequence of of lemniscate trees sampled uniformly at random, and let $X_n$ denote the number of vertices in $T_n$ of out-degree two. Let $\mu_n$ and $\sigma_n$ denote the mean and standard deviation of $X_n$. Then $$\mu_n=\left(1-\dfrac{2}{\pi}\right)n+O(1)\text{ and }\sigma_n^2=\left(\dfrac{4}{\pi^2}+\dfrac{2}{\pi}-1\right)n+O(1).$$

Moreover, $\sigma_n^{-1}(X_n-\mu_n)$ converges in distribution to a standard Gaussian random variable as $n\to\infty$.

\end{theorem}

In 2018, A.~Frolova~et.~al.~\cite{FrolovaKhavinsonVasil'ev} explored the construction of lemniscate trees by a means they called polynomial fireworks. In this process, a single zero $z_0$ of a complex polynomial $p(z)$ is replaced by the zeros of a second polynomial $q(z)$. That is, $p(z)\mapsto (z-z_0)^{-1}p(z)q(z)$. They prove the following result regarding the effect of this process on the lemniscate tree.

\begin{theorem}
Let $p(z)$ be a lemniscate generic complex polynomial, and let $z_0\in\mathbb{C}$ be one of the zeros of $p$. If the zeros of $q$ are all sufficiently close to $z_0$, then the lemniscate tree of the polynomial $(z-z_0)^{-1}p(z)q(z)$ is obtained by appending the lemniscate tree of $q$ to the leaf of the lemniscate tree of $p$ corresponding to $z_0$, and merely extending the other leaves the appropriate length.
\end{theorem}

In 2015, T.~J.~Richards~\cite{Richards1} expanded the definition of the lemniscate tree by taking into account not just the inclusion relation of one critical lemniscate or zero lying in the bounded component of the complement of the other, but also i) the critical value associated with each critical point of the underlying polynomial, and ii) the rotational orientation of each interior critical lemniscate. This notion of the configuration of critical lemniscates also accommodated lemniscate non-generic polynomials. Let $U$ denote the collection of equivalence classes (modulo precomposition with an affine map) of complex polynomials with a prescribed list of critical values, and let $V$ denote the collection of critical lemniscate configurations (roughly lemniscate trees with critical value and rotation data as described above).

\begin{theorem}
The map $\Pi:U\to V$ which takes a polynomial to its critical lemniscate configuration is a bijection.
\end{theorem}

\section{Fingerprints of Shapes and Conformal Equivalence}
\label{sect: Fingerprints of Shapes and Conformal Equivalence.}

One of the reasons for the recent burgeoning interest in the lemniscates of complex polynomials is their potential role in the field of shape analysis. Define a shape $\Gamma$ to be a simple, smooth, closed curve in the plane, with bounded interior region $\Omega_-$ and unbounded exterior region $\Omega_+$. Let $\mathbb{D}$ denote the unit disk, and let $\mathbb{D}_+$ denote the region $\hat{\mathbb{C}}\setminus cl(\mathbb{D})$. Let $\Phi_-:\mathbb{D}\to\Omega_-$ and $\Phi_+:\mathbb{D}_+\to\Omega_+$ be analytic bijections (whose existence is quaranteed by the Riemann mapping theorem). Adopt also the normalization $\Phi_+(\infty)=\infty$ and ${\Phi_+}(\infty)>0$. Since $\Gamma$ is smooth, $\Phi_+$ and $\Phi_-$ may be extended smoothly to the boundary of their domains. The fingerprint of $\Gamma$ is defined to be the self-map of the unit circle $\tau:\mathbb{T}\to\mathbb{T}$ defined by $k={\Phi_+}^{-1}\circ\Phi_-$. The map from shapes (modulo precomposition with affine transformations) to orientation-preserving diffeomorphisms of $\mathbb{T}$ (modulo precomposition with an automorphism of the disk) is known to be a bijection. The problem of recovering a shape from its fingerprint has been explored numerically, and several algorithms have been developed (see~\cite{EbenfeltKhavinsonShapiro} and the discussion contained therein for these results). In the special case that the shape is a proper non-singular polynomial lemniscate, the corresponding fingerprint has a particularly nice form. In 2011, P.~Ebenfelt~et.~al~\cite{EbenfeltKhavinsonShapiro} showed the following.

\begin{theorem}
Let $p(z)$ be a degree $n$ complex polynomial, and suppose that the level set $\Lambda_1(p)$ has a single, non-singular component. Then the fingerprint of $\Lambda_1(p)$ is an $n^{\text{th}}$ root of a degree-$n$ finite Blaschke product. Conversely, every $n^{\text{th}}$ root of a degree-$n$ finite Blaschke product is the fingerprint for some such lemniscate.

\end{theorem}

In 2018, A.~Frolova~et.~al.~\cite{FrolovaKhavinsonVasil'ev} studied the fingerprints of smooth shapes, viewing them as smooth increasing bijections $\tau:[0,2\pi]\to[0,2\pi]$ (modulo the identification $0\sim2\pi$), rather than self-maps of the unit circle. They proved the following.

\begin{theorem}
Let $p(z)$ be a degree $n$ complex polynomial, and suppose that the lemniscate $\Lambda_1(p)$ has a single, non-singular component. Then the fingerprint of $\Lambda_1(p)$ has an even number of inflection points, at least $2$ and at most $4n-2$.
\end{theorem}

Suppose again that $\Gamma=\Lambda_1(p)$ is a lemniscate of a degree $n$ complex polynomial $p$, with a single, non-critical component (that is, all of the critical values of $p$ have magnitude less than $1$, see~\cite{EbenfeltKhavinsonShapiro} or~\cite{Younsi} for details). As before, let $\Omega_+$ denote the region exterior to $\Gamma$. Then the exterior Riemann map $\Phi_+:\mathbb{D}_+\to\Omega_+$ may be taken to be $\Phi_+(z)=p(z)^{1/n}$. Let $B(z)$ be degree $n$ Blaschke product whose $n^{th}$ root is a fingerprint for $\Gamma$ (whose existence is shown in~\cite{EbenfeltKhavinsonShapiro}, as mentioned above). Then taking $n^{\text{th}}$ powers, we have the equation $B=p\circ\Phi_-$ on $\mathbb{D}$. The interesting direction is the converse (also following from~\cite{EbenfeltKhavinsonShapiro}).

\begin{theorem}\label{thm: Conformal equivalence first.}
For any finite Blaschke product $B$, there is a complex polynomial $p$ with the same degree as $B$, and an injective analytic map $\varphi:\mathbb{D}\to\mathbb{C}$ for which $B=p\circ\varphi$ on $\mathbb{D}$.
\end{theorem}

In general, if $f$ is an analytic (or later, meromorphic) function on a domain $E\subset\mathbb{C}$, and there is an injective analytic map $\varphi:E\to\mathbb{C}$ and an analytic (or meromorphic) map $g$ with domain $\varphi(E)$ such that $f=g\circ\varphi$ on $E$, then $g$ is said to be a conformal model $f$ on $E$. With this notation, Theorem~\ref{thm: Conformal equivalence first.} states that a finite Blaschke product $B$ has a polynomial conformal model $p$ on $\mathbb{D}$, with $\deg(p)=\deg(B)$. Theorem~\ref{thm: Conformal equivalence first.} was also proved by different means by T.~J.~Richards~\cite{Richards1} in 2015. In 2016, Richards~\cite{Richards2} extended this result to general analytic functions which are analytic across the boundary of the unit disk, though this time with no control on the degree of the polynomial.

\begin{theorem}\label{thm: Conformal analysis for disk functions.}
Let $f$ be a function which is analytic on an open set containing the closed unit disk. Then $f$ has a polynomial conformal model on $\mathbb{D}$.
\end{theorem}

In 2017, T.~J.~Richards and M.~Younsi~\cite{RichardsYounsi1} gave a version of Theorem~\ref{thm: Conformal analysis for disk functions.} for meromorphic functions, in which they were also able to recover control over the degree of the polynomial $p$ (now a rational function $q$), subject to a condition on the behavior of the function $f$ on the boundary of the disk.

\begin{theorem}
Let $f$ be meromorphic function on an open set containing the closed unit disk, such that i) $f$ has no critical points on $\mathbb{T}$, and ii) $f(\mathbb{T})$ is a Jordan curve, whose bounded face contains $0$. Suppose without loss of generality that the number of zeros $m$ of $f$ lying in $\mathbb{D}$ is greater than or equal to the number of poles $n$ of $f$ lying in $\mathbb{D}$. Then there is a rational function $q$ and an injective analytic map $\varphi:\mathbb{D}\to\mathbb{C}$ such that the following hold.

\begin{itemize}
\item
$f=q\circ\varphi$ on $\mathbb{D}$.

\item
$q$ has $m$ zeros, all of which lie in $\varphi(\mathbb{D})$. $q$ has $n$ poles lying in $\varphi(\mathbb{D})$, and the only pole of $q$ not lying in $\varphi(\mathbb{D})$ is at $\infty$, with multiplicity $n-m$ (if that quantity is non-zero).
\end{itemize}
\end{theorem}

Richards and Younsi also established a negative result regarding the degree of the polynomial conformal model for an analytic disk function $f$ to the effect that the minimal degree of a polynomial conformal model for $f$ on $\mathbb{D}$ cannot be determined by the degree of non-injectivity of $f$ on $\mathbb{D}$ (that is, how many-to-one $f$ is on $\mathbb{D}$).

\begin{theorem}
For any $n\geq2$, there is a function $f_n$ which is analytic on an open set containing the closed unit disk for which the following holds.

\begin{itemize}
\item
$f_n$ is at most $2$-to-$1$ on $\mathbb{D}$.

\item
$f_n$ has no polynomial conformal model with degree $\leq n$.

\end{itemize}

\end{theorem}

In 2016, M.~Younsi~\cite{Younsi} showed that a rational function may be found which is simultaneously conformally equivalent to any two prescribed finite Blaschke products $A$ and $B$, on $\mathbb{D}$ and $\mathbb{D}_+$ respectively.

\begin{theorem}
Let $A$ and $B$ be finite Blaschke products. There is a rational function $q(z)$ for which the lemniscate $\Gamma=\Lambda_1(q)$ is a single, non-critical component, and for which $A=q\circ\Phi_-$ on $\mathbb{D}$ and $B=q\circ\Phi_+$ on $\mathbb{D}_+$.
\end{theorem}

In 2019, T.~J.~Richards and M.~Younsi~\cite{RichardsYounsi2} gave a first constructive result, describing an explicit construction for the polynomial conformal model for finite Blaschke products of degree at most $3$. They also gave the following formula for the polynomial conformal model $p$ and associated injective analytic map $\varphi$ for a finite Blaschke product of arbitrarily high degree, whose zeros are evenly distributed on a circle centered at the origin.

\begin{theorem}
Let $c\in\mathbb{D}$, $\lambda\in\mathbb{T}$, and $n\geq1$ be chosen. Define $B(z)=\lambda\dfrac{z^n-c^n}{1-\bar{c}^nz^n}$. Then $B$ has polynomial conformal model $p(z)=\lambda\left(|c|^{2n}-1\right)z^n-\lambda c^n$. Setting $\psi(z)=\dfrac{e^{i\pi/n}z}{\sqrt[n]{1-\bar{c}^nz^n}}$, $\psi^{-1}$ is an injective analytic map on $\mathbb{D}$, and $B=p\circ\psi^{-1}$ on $\mathbb{D}$.

\end{theorem}

In 2013, T.~J.~Richards~\cite{RichardsLowtherSpeyer} posted Theorem~\ref{thm: Conformal analysis for disk functions.} as a conjecture on the website \textit{math.stackexchange.com}. As noted above, Richards published a proof for this result in 2016. Before that, also in 2013, users G.~Lowther and D.~Speyer provided a proof for a more general result (also on~\cite{RichardsLowtherSpeyer}), where the disk $\mathbb{D}$ is replaced with an arbitrary compact set. As we wish to include this more general result, and the proof has not appeared in a peer-reviewed source in the intervening years, we will present the proof here, with an extension also to meromorphic functions. It should be emphasized that the proof presented here is essentially that of Lowther and Speyer, the only non-trivial changes being those necessary to accommodate meromorphic rather than analytic functions.

\begin{theorem}\label{thm: To prove.}
Let $K\subset\mathbb{C}$ be compact, and let $f$ be a function which is meromorphic on $K$. Then there is an injective analytic function $\varphi:K\to\mathbb{C}$, and a rational function $q$ such that $f=q\circ\varphi$ on $K$.
\end{theorem}

Mirroring the work of Lowther and Speyer, we will make use of the following lemmas.

\begin{lemma}\label{lem: Lemma 1.}
Let $f$, $\{g_n\}_{n=1}^\infty$ be non-constant analytic functions on an open set $U\subset\mathbb{C}$, and assume that $g_n\to f$ uniformly on $U$. Let $K\subset U$ be compact, and suppose the following holds.
\begin{itemize}
\item
If $z\in K$ is a critical point of $f$ with multiplicity $m\geq 1$, then for each $n\geq 1$, and each $j\in\{0,1,\ldots,m\}$, ${g_n}^{(j)}(z)=f^{(j)}(z)$.

\end{itemize}

Then for all sufficiently large $n$, there is an injective analytic map $\varphi_n:K\to U$ such that $f=g_n\circ\varphi_n$ on $K$.
\end{lemma}

\begin{lemma}\label{lem: Lemma 2.}
Let $f$ be analytic and non-constant on a compact set $K\subset\mathbb{C}$. There is an open neighborhood $U$ of $K$ and a sequence of rational functions $q_n$ having no poles in $U$ and having only a single pole in each component of $U^c$ for which $q_n\to f$ uniformly on $U$, and such that if $z\in K$ is a critical point of $f$ with multiplicity $m\geq 1$, then for each $n\geq 1$, and each $j\in\{0,1,\ldots,m\}$, ${q_n}^{(j)}(z)=f^{(j)}(z)$.
\end{lemma}

\begin{proof}[Proof of Theorem~\ref{thm: To prove.}]
Let $K\subset\mathbb{C}$ be compact, and let $\mathcal{O}$ be an open set containing $K$. Let $f:\mathcal{O}\to\hat{\mathbb{C}}$ be meromorphic. By replacing $\mathcal{O}$ with a slightly smaller open set, still containing $K$, we may assume that $f$ is meromorphic on the closure of $\mathcal{O}$ (and thus has only finitely many poles on $\mathcal{O}$), with no critical points or poles on $\partial\mathcal{O}$. Let $w_1,w_2,\ldots,w_M\in\mathcal{O}$ be the poles of $f$ in $\mathcal{O}$, with multiplicities $m_1,m_2,\ldots,m_M\in\mathbb{N}$.

Around each pole $w_k$, there is an open neighborhood $E_k$ such that for some analytic bijection $\psi_k:E_k\to\mathbb{D}$, with $\psi_k(w_k)=0$, $f(z)=\dfrac{c_k}{\psi_k(z)^{m_k}}$ on $E_k$. By i) dividing $f$ by a large enough constant, ii) reducing the neighborhoods $E_k$ as necessary, and iii) making the appropriate choice of the maps $\psi_k$, we may assume without loss of generality that each $c_k=1$. Define $\mathcal{O}_2=\mathcal{O}\setminus\ds\bigcup E_k$.

By Lemma~\ref{lem: Lemma 2.}, we may choose a neighborhood $U$ of $cl(\mathcal{O}_2)$, and a sequence of rational functions $\{q_n\}$ which interpolates the values and derivative data at each critical point of $f$ in $\mathcal{O}_2$. By Lemma~\ref{lem: Lemma 1.}, for sufficiently large $n$, there is an injective analytic map $\varphi_n:\mathcal{O}_2\to U$ such that $f=q_n\circ\varphi_n$ on $\mathcal{O}_2$. Let $n_0$ denote the smallest such (or any such) value of $n$. Set $q=q_{n_0}$ and $\varphi=\varphi_{n_0}$.

For each $1\leq k\leq M$, let $\widetilde{E_k}$ denote the bounded region bounded by $\varphi\left(\partial E_k\right)$. Since each $E_k$ contained a single distinct pole of $f$ of multiplicity $m_k$, each $\widetilde{E_k}$ contains a single distinct pole of $q$ of multiplicity $m_k$. Thus since $|q|=1$ on $\partial\widetilde{E_k}$ (since $|f|=1$ on $\partial E_k$), we may choose $\widetilde{\psi_k}:\widetilde{E_k}\to\mathbb{D}$ to be a Riemann map for $\widetilde{E_k}$ for which $$q(z)=\dfrac{1}{\widetilde{\psi_k}(z)^{m_k}}\text{ on }\widetilde{E_k}.$$ Thus if we extend $\varphi$ from $\mathcal{O}_2$ to $\mathcal{O}$ by $\varphi=\widetilde{\psi_k}^{-1}\circ\psi_k$ on $E_k$, then $\varphi$ is continuous, thus analytic across the boundary of $E_k$, and $f=q\circ\varphi$ on all of $\mathcal{O}$.

\end{proof}

\begin{proof}[Proof of Lemma~\ref{lem: Lemma 1.}]
By restricting $U$ to a small enough open set containing $K$, we may assume without loss of generality $f$ has no critical points in $U\setminus K$. Fix some $z_0\in U$. Our first goal is to show that there is a small neighborhood $V_0$ of $z_0$, and a sequence of injective analytic functions $\psi_n:V_0\to U$ with $\psi_n(z)\to z$ uniformly on $V_0$, and $f=g_n\circ\psi_n$ on $V_0$ for all sufficiently large $n$.

Suppose first that $f'(z_0)\neq0$. By rescaling $f$ and $p_n$ if necessary, we can assume that $f'(z_0)=1$. Choose some $r>0$ such that the closed ball $cl(B(z_0;r))$ is contained in $U$, and $\Re(f')>1/2$ on $cl(B(z_0;r))$. Then by uniform convergence, $\Re({g_n}')>1/2$ on $cl(B(z_0;r))$ for sufficiently large values of $n$. This implies that for $z,z'\in cl(B(z_0;r))$, $$\Re\left(\dfrac{g_n(z)-g_n(z')}{z-z'}\right)>\dfrac{1}{2},$$ so $g_n$ is injective with $|{g_n}'|\geq 1/2$ on $cl(B(z_0;r))$. It follows therefore that $g_n\left(B(z_0;r)\right)$ contains $B(g_n(z_0);r/2)$, so $g_n\left(B(z_0;r)\right)\supset B\left(f(z_0);r/3\right)$ for sufficiently large values of $n$ (again by the uniform convergence of $g_n\to f$), and there is a unique analytic inverse ${g_n}^{-1}:B\left(f(z_0);r/3\right)\to B(z_0;r)$ with $g_n\circ{g_n}^{-1}(z)=z$ (by the inverse function theorem). Choosing the open neighborhood $V_0$ of $z_0$ small enough that $f(V_0)\subset B(f(z_0);r/3)$, then defining $\psi_n:V_0\to U$ by $\psi_n={g_n}^{-1}\circ f$ satisfies the requirements.

Suppose now that $f'(z_0)=0$. Subtract a constant if necessary from $f$, and the same constant from each $g_n$, to ensure that $f(z)=(z-z_0)^mh(z)$ for some $m\geq2$ and for an analytic function $h:U\to\mathbb{C}$ with $g(z_0)\neq0$. By assumption, each $g_n=(z-z_0)^mh_n(z)$ for analytic functions $h_n:U\to\mathbb{C}$ with $h_n\to h$ uniformly on $U$. Then on a neighborhood of $z_0$, $h$ and each $h_n$ is nonzero for sufficiently large $n$. Hence, we can take $m^{\text{th}}$ roots to obtain analytic functions $\widetilde{f}(z)=(z-z_0)h^{1/m}$ and $\widetilde{g_n}(z)=(z-z_0)h_n(z)^{1/m}$. Moreover, provided that we take consistent $m^{\text{th}}$ roots, then $\widetilde{g_n}\to f$ uniformly on the neighborhood of $z_0$. Therefore by the first case above, there exists an open neighborhood $V_0$ of $z_0$ and analytic functions $\psi_n:V_0\to U$ with $\psi_n(z)\to z$ uniformly, with $\widetilde{f}=\widetilde{g_n}\circ\psi_n$ on $V_0$. Taking $m^{\text{th}}$ powers, we have $f=g_n\circ\psi_n$ on $V_0$.

By compactness of $K$ and the fact that the analytic functions $\psi_n$ exist locally as shown above, there is a finite open cover $\{B_1,\ldots,B_N\}$ of $K$ for which the $B_k$ are open balls in $U$, and sequences of analytic functions $\{\psi_{k,n}\}_{n=1}^\infty$ satisfying $g_n\circ\psi_{k,n}=f$ on $B_k$, and $\psi_{k,n}(z)\to z$ on uniformly $B_k$.

However, whenever $B_k$ and $B_l$ have non-empty intersection, since $f$ is non-constant, its derivative will be non-zero at some point $z_0\in B_k\cap B_l$, and without loss of generality, suppose that $f'(z_0)=1$. Then by the uniform convergence, there is an open neighborhood $\hat{B}$ of $z_0$ on which $\Re({g_n}')\geq 1/2$ for sufficiently large $n$, so that $g_n$ is injective on $\hat{B}$. Since $g_n\circ\psi_{k,n}=f=g_n\circ\psi_{l,n}$ on $\hat{B}$, and $g_n$ is injective on $\hat{B}$, it follows that $\psi_{k,n}=\psi_{l,n}$ on $\hat{B}$ (and thus on all of $B_k\cap B_l$). Thus setting $V=\ds\bigcup B_k$, we have have analytic functions $\psi_n:V\to U$ (setting $\psi_n=\psi_{k,n}$ on $B_k$), with $f=g_n\circ\psi_n$ on $V$, and $\psi_n(z)\to z$ uniformly.

It only remains to show that $\psi_n$ is injective on all of $V$. Let $\widehat{B_1},\ldots,\widehat{B_t}$ be open balls covering $K$, whose closures are contained in $V$. Let $n$ be chosen large enough so that $\psi_n$ is injective on each $\widehat{B_k}$, and set $\widehat{K}=\ds\bigcup\widehat{B_k}$. By compactness, there is an $\epsilon>0$ such that for each $z,w\in\widehat{K}$, if $0<|z-w|<\epsilon$, $z$ and $w$ lie in some common $\widehat{B_k}$, so that $\psi_n(z)\neq\psi_n(w)$. Additionally, since $\psi_n(z)\to z$ uniformly on $V$, we may also require that $|\psi_n(z)-z|<\epsilon/2$ on $V$. Therefore, for any distinct $z,w\in V$, if $|z-w|<\epsilon$, $\psi_n(z)\neq \psi(w)$, and if $|z-w|\geq\epsilon$, $|\psi_n(z)-\psi_n(w)|\geq\epsilon-|z-w|>0$ (by the reverse triangle inequality). Thus $\psi_n$ is injective of $V$.

\end{proof}

\begin{proof}[Proof of Lemma~\ref{lem: Lemma 2.}]
To begin, let an open, bounded set $U$ be chosen which contains $K$, and such that $f$ is analytic on the closure of $U$. By Runge's theorem, we may find a sequence of rational functions $\left\{\widehat{q_n}\right\}$ which converge uniformly to $f$ on $U$, and such that each $\widehat{q_n}$ i) is analytic on $U$ and ii) has at most one pole in each component of $U^c$. Let $z_1,z_2,\ldots,z_M\in K$ be the critical points of $f$ in $K$, with multiplicities $m_1,m_2,\ldots,m_M\geq1$. Define $N=\ds\sum(m_k+1)$. By Lagrange interpolation, for each $n\in\mathbb{N}$, there is a unique polynomial $r_n$ of degree $N-1$ such that for each $k\in\{1,2,\ldots,M\}$ and each $j\in\{0,1,\ldots,m_k\}$, $\widehat{q_n}^{(j)}(z_k)-{r_n}^{(j)}(z_k)=f^{(j)}(z_k)$. We wish to show that $q_n=\widehat{q_n}-r_n\to f$ uniformly on $U$. Since $\widehat{q_n}\to f$ uniformly on $U$, it suffices to show that $r_n\to 0$ uniformly on $U$.

The coefficients of $r_n$ depend linearly on the $N$ quantities $\widehat{q_n}^{(j)}(z_k)$. These coefficients do not depend on $n$. Thus it suffices to show that each $\widehat{q_n}^{(j)}(z_k)$ approaches zero as $n\to\infty$. Fix some $k\in\{1,2,\ldots,M\}$. For $j=0$, observe that since $\widehat{q_n}\to f$ uniformly on $U$, ${r_n}^{(0)}(z_k)=f(z_k)-\widehat{q_n}(z_k)\to0$. Let $\gamma_k$ be a small circle around $z_k$, on which $f$ is analytic, and which does not enclose or contain any other $z_l$. For $j\in\{1,2,\ldots,m_k\}$, $$\widehat{q_n}^{(j)}(z_k)=\widehat{q_n}^{(j)}(z_k)-f^{(j)}(z_k)=\dfrac{j!}{2\pi i}\ds\oint_{\gamma_k}\dfrac{\widehat{q_n}(z)-f(z)}{(z-z_k)^{j+1}}dz.$$ Since $\widehat{q_n}\to f$ uniformly on $U$, this integral approaches $0$ as $n\to\infty$.

\end{proof}

\bibliographystyle{plain}
\bibliography{refs.bib}

\end{document}